\documentclass[11pt, a4paper]{article}
\usepackage[T1]{fontenc}
\usepackage{amsfonts}
\usepackage{amsmath}
\usepackage{amssymb}
\usepackage{amsthm}
\usepackage{bbm}
\usepackage{bm}
\usepackage{mathrsfs}
\usepackage{verbatim}
\usepackage{setspace}
\usepackage{enumitem}
\theoremstyle{plain}

\newtheorem{theorem}{Theorem}[section]
\newtheorem{proposition}[theorem]{Proposition}
\newtheorem{lemma}[theorem]{Lemma}
\newtheorem{corollary}[theorem]{Corollary}

\theoremstyle{definition}

\newtheorem{remark}[theorem]{Remark}

\theoremstyle{remark}

\renewenvironment{thebibliography}[1]{%
\begin{oldthebibliography}{#1}%
\setlength{\baselineskip}{.9em}
\linespread{.9}
\small
\setlength{\parskip}{0ex}%
\setlength{\itemsep}{.1em}%
}%
{%
\end{oldthebibliography}%
}

\newcommand{\q}{\quad}

\newcommand{\eps}{\varepsilon}
\newcommand{\idmat}{\mathbbm{1}} 
\newcommand{\N}{\mathbb{N}}
\newcommand{\Z}{\mathbb{Z}}

\newcommand{\R}{\mathbb{R}}

\renewcommand{\S}{\mathbb{S}}

\renewcommand{\L}{\mathbb{L}}
\newcommand{\cF}{\mathcal{F}}

\newcommand{\cP}{\mathcal{P}}
\newcommand{\cQ}{\mathcal{Q}}

\newcommand{\cE}{\mathcal{E}}

\newcommand{\bD}{\mathbf{D}}

\newcommand{\br}[1]{\langle #1 \rangle}

\DeclareMathOperator{\tr}{trace}

\newcommand{\as}{\mbox{-a.s.}}
\renewcommand{\ae}{\mbox{-a.e.}}

\newcommand{\1}{\mathbf{1}}

\numberwithin{equation}{section}

\begin{document}

\title{\vspace{-1.1 cm} Weak Approximation of $G$-Expectations\\
\date{March 2, 2011}
\author{
  Yan Dolinsky%
  \thanks{
  ETH Zurich, Dept.\ of Mathematics, CH-8092 Zurich, \texttt{yan.dolinsky@math.ethz.ch}
  }
  \and
  Marcel Nutz%
  \thanks{
  ETH Zurich, Dept.\ of Mathematics, CH-8092 Zurich, \texttt{marcel.nutz@math.ethz.ch}
  }
  \and
  H.\ Mete Soner%
  \thanks{
  ETH Zurich, Dept.\ of Mathematics, CH-8092 Zurich, and Swiss Finance Institute, \texttt{mete.soner@math.ethz.ch}
  }
 }
}
\maketitle \vspace{-1.5em}

\begin{abstract}
We introduce a notion of volatility uncertainty in discrete time and define the corresponding analogue of Peng's $G$-expectation.
In the continuous-time limit, the resulting sublinear expectation converges weakly to the $G$-expectation. This can be seen as a Donsker-type result for the $G$-Brownian motion.
\end{abstract}

{\small
\noindent \emph{Keywords} $G$-expectation, volatility uncertainty, weak limit theorem


\noindent \emph{AMS 2000 Subject Classifications}
60F05, 
60G44, 
91B25, 
91B30  

\noindent \emph{JEL Classifications} G13, G32
}\\

\noindent \emph{Acknowledgements}
Research supported by
European Research Council Grant 228053-FiRM, Swiss National Science Foundation
Grant PDFM2-120424/1, Swiss Finance Institute and ETH Foundation.

\section{Introduction}

The so-called $G$-expectation~\cite{Peng.07, Peng.08, Peng.10} is a nonlinear expectation
advancing the notions of backward stochastic differential equations (BSDEs)~\cite{PardouxPeng.90} and \mbox{$g$-expectations}~\cite{Peng.97}; see also~\cite{CheriditoSonerTouziVictoir.07, SonerTouziZhang.2010bsde} for a related theory of second order BSDEs.
A $G$-expectation $\xi\mapsto \cE^G(\xi)$ is a sublinear function which maps random variables $\xi$ on the canonical space $\Omega=C([0,T];\R)$ to the real numbers. The symbol $G$ refers to a given function $G: \R\to\R$ of the form
\[
  G(\gamma)=\frac{1}{2}(R\gamma^+-r\gamma^-)=\frac{1}{2}\sup_{a\in [r,R]} a\gamma,
\]
where $0\leq r\leq R<\infty$ are fixed numbers.
More generally, the interval $[r,R]$ is replaced by a set $\bD$ of nonnegative matrices in the multivariate case. The extension to a random set $\bD$ is studied in~\cite{Nutz.10Gexp}.

The construction of $\cE^G(\xi)$ runs as follows.
When $\xi=f(B_T)$, where $B_T$ is the canonical process at time $T$ and $f$ is a sufficiently regular function, then $\cE^G(\xi)$ is defined to be the initial value $u(0,0)$ of the solution of the nonlinear backward heat equation $-\partial_t u - G(u_{xx})=0$ with terminal condition $u(\cdot,T)=f$.
The mapping $\cE^G$ can be extended to random variables of the form $\xi=f(B_{t_1},\dots,B_{t_n})$ by a stepwise evaluation of the PDE and then to the completion $\L^1_G$ of the space of all such random variables. The space $\L^1_G$ consists of so-called quasi-continuous functions and contains in particular all bounded continuous functions on $\Omega$; however, not all bounded measurable functions are included (cf.~\cite{DenisHuPeng.2010}).
While this setting is not based on a single probability measure, the so-called $G$-Brownian motion is given by the canonical process $B$ ``seen'' under $\cE^G$ (cf.\ \cite{Peng.10}). It reduces to the standard Brownian motion if $r=R=1$ since $\cE^G$ is then the (linear) expectation under the Wiener measure.

In this note we introduce a discrete-time analogue of the $G$-expectation and we prove a convergence result which resembles Donsker's theorem for the standard Brownian motion; the main purpose is to provide additional intuition for $G$-Brownian motion and volatility uncertainty. Our starting point is the dual view on $G$-expectation via volatility uncertainty~\cite{DenisHuPeng.2010, DenisMartini.06}: We consider the representation
\begin{equation}\label{eq:GexpIntro}
  \cE^G(\xi)=\sup_{P\in\cP} E^P[\xi],
\end{equation}
where $\cP$ is a set of probabilities on $\Omega$ such that under any $P\in\cP$, the canonical process $B$ is a martingale with volatility $d\br{B}/dt$ taking values in $\bD=[r,R]$, $P\times dt$-a.e. Therefore, $\bD$ can be understood as the domain of (Knightian) volatility uncertainty and $\cE^G$ as the corresponding worst-case expectation. In discrete-time, we translate this to uncertainty about the conditional variance of the increments. Thus we define a sublinear expectation $\cE^n$ on the $n$-step canonical space in the spirit of~\eqref{eq:GexpIntro}, replacing $\cP$ by a suitable set of martingale laws. A natural push-forward then yields a sublinear expectation on $\Omega$, which we show to converge weakly to $\cE^G$ as $n\to\infty$, if the domain $\bD$ of uncertainty is scaled by $1/n$ (cf.~Theorem~\ref{th:limit}). The proof relies on (linear) probability theory; in particular, it does not use the central limit theorem for sublinear expectations~\cite{Peng.10, Peng.10CLT}. The relation to the latter is nontrivial since our discrete-time models do not have independent increments. We remark that quite different approximations of the $G$-expectation (for the scalar case) can be found in discrete models for financial markets with transaction costs~\cite{Kusuoka.95} or illiquidity~\cite{DolinskySoner.11}.

The detailed setup and the main result are stated in Section~\ref{se:main}, whereas the proofs and some ramifications are given in Section~\ref{se:proof}.

\section{Main Result}\label{se:main}

We fix the dimension $d\in\N$ and denote by $|\cdot|$ the Euclidean norm on $\R^d$. Moreover, we denote by $\S^d$ the space of $d\times d$ symmetric matrices and by $\S^d_+$ its subset of nonnegative definite matrices. We fix a nonempty, convex and compact set
$\bD\subseteq \S^d_+$; the elements of $\bD$ will be the possible values of our volatility processes.

\paragraph{Continuous-Time Formulation.}
Let $\Omega=C([0,T];\R^d)$ be the space of $d$-dimensional continuous paths $\omega=(\omega_t)_{0\leq t\leq T}$ with time horizon $T\in(0,\infty)$, endowed with the uniform norm $\|\omega\|_\infty=\sup_{0\leq t\leq T} |\omega_t|$. We denote by $B=(B_t)_{0\leq t\leq T}$ the canonical process $B_t(\omega)=\omega_t$ and by $\cF_t:=\sigma(B_s,\, 0\leq s\leq t)$ the canonical filtration.
A probability measure $P$ on $\Omega$ is called a \emph{martingale law} if $B$ is a $P$-martingale and $B_0=0$ $P$-a.s. (All our martingales will start at the origin.) We set
\[
  \cP_\bD=\big\{P\mbox{ martingale law on }\Omega:\,d\br{B}_t/dt \in \bD,\; P\times dt\ae\big\},
\]
where $\br{B}$ denotes the matrix-valued process of quadratic covariations. We can then define the sublinear expectation
\[
  \cE_\bD(\xi):=\sup_{P\in \cP_\bD} E^P[\xi]\quad \mbox{for any random variable }\xi: \Omega\to\R
\]
such that $\xi$ is $\cF_T$-measurable and $E^P|\xi|<\infty$ for all $P\in\cP_\bD$. The mapping $\cE_\bD$ coincides with the $G$-expectation (on its domain $\L^1_G$) if $G:\S^d\to\R$ is (half) the support function of $\bD$; i.e.,
$G(\Gamma)=\sup_{A\in\bD} \tr (\Gamma A)/2$. Indeed, this follows from~\cite{DenisHuPeng.2010} with an additional density argument as detailed in Remark~\ref{rk:GexpIdentification} below.

\paragraph{Discrete-Time Formulation.}
Given $n\in\N$, we consider $(\R^d)^{n+1}$ as the canonical space of $d$-dimensional paths in discrete time $k=0,1,\dots,n$.
We denote by $X^n=(X^n_k)_{k=0}^n$ the canonical process defined by $X^n_k(x)=x_k$ for $x=(x_0,\dots,x_n)\in(\R^d)^{n+1}$. Moreover,
$\cF^n_k=\sigma (X^n_i,\, i=0,\dots,k)$ defines the canonical filtration $(\cF^n_k)_{k=0}^n$. We also introduce $0\leq r_\bD\leq R_\bD<\infty$ such that $[r_\bD,R_\bD]$ is the spectrum of $\bD$; i.e.,
\[
  r_\bD=\inf_{\Gamma\in\bD} \|\Gamma^{-1}\|^{-1}\quad\mbox{and}\quad R_\bD=\sup_{\Gamma\in\bD} \|\Gamma\|,
\]
where $\|\cdot\|$ denotes the operator norm and we set $r_\bD:=0$ if $\bD$ has an element which is not invertible. We note that $[r_\bD,R_\bD]=\bD$ if $d=1$.
Finally, a probability measure $P$ on $(\R^d)^{n+1}$ is called a martingale law if $X^n$ is a $P$-martingale and $X^n_0=0$ $P$-a.s.
Denoting by $\Delta X^n_k=X^n_k-X^n_{k-1}$ the increments of $X^n$, we can now set
\[
  \cP^n_\bD=
    \left\{\hspace{-4pt}
      \begin{array}{l}
        P\mbox{ martingale law on }(\R^d)^{n+1}:\,\mbox{ for }k=1,\dots,n,\\[.2em]
        E^P[\Delta X_k^n (\Delta X_k^n)'|\cF^n_{k-1}] \in \bD\mbox{ and }
        d^2r_\bD\leq |\Delta X^n_k|^2\leq d^2R_\bD,\;
        P\as\hspace{-4pt}
      \end{array}
    \right\},
\]
where prime ($'$) denotes transposition. Note that $\Delta X^n_k$ is a column vector, so that $\Delta X^n (\Delta X^n)'$ takes values in $\S^d_+$. We introduce the sublinear expectation
\[
  \cE^n_\bD(\psi):=\sup_{P\in \cP^n_\bD} E^P[\psi]\quad \mbox{for any random variable }\psi: (\R^d)^{n+1}\to\R
\]
such that $\psi$ is $\cF^n_n$-measurable and $E^P|\psi|<\infty$ for all $P\in\cP^n_\bD$, and we think of
$\cE^n_\bD$ as a discrete-time analogue of the $G$-expectation.

\begin{remark}
  The second condition in the definition of $\cP^n_\bD$ is motivated by the desire to generate the volatility uncertainty by a \emph{small} set of scenarios; we remark that the main results remain true if, e.g., the lower bound $r_\bD$ is omitted and the upper bound $R_\bD$ replaced by any other condition yielding tightness. Our bounds are chosen so that
  \[
    \cP^n_\bD=\big\{P\mbox{ martingale law on }(\R^d)^{n+1}: \Delta X^n (\Delta X^n)'\in \bD, \;P\as\big\}\quad\mbox{if }d=1.
  \]
\end{remark}

\paragraph{Continuous-Time Limit.}
To compare our objects from the two formulations, we shall extend any discrete path $x\in(\R^d)^{n+1}$ to a continuous path $\widehat{x}\in \Omega$ by linear interpolation. More precisely, we define the  interpolation operator
\begin{align*}
  \widehat{}\,:\quad &(\R^d)^{n+1}\to \Omega,\quad x=(x_0,\dots,x_n)\mapsto \widehat{x}=(\widehat{x}_t)_{0\leq t\leq T},\quad \mbox{where}\nonumber\\
  &\widehat{x}_t:=([nt/T] + 1 -nt/T)x_{[nt/T]} + (nt/T-[nt/T])x_{[nt/T]+1}
\end{align*}
and $[y]:=\max\{m\in\Z: m\leq y\}$ for $y\in\R$. In particular, if $X^n$ is the canonical process on $(\R^d)^{n+1}$ and $\xi$ is a random variable on $\Omega$, then
$\xi(\widehat{X^n})$ defines a random variable on $(\R^d)^{n+1}$. This allows us to define the following push-forward of $\cE^n_\bD$ to a continuous-time object,
\[
  \widehat{\cE}^n_\bD(\xi):=\cE^n_\bD(\xi(\widehat{X^n})) \quad \mbox{for}\quad\xi: \Omega\to\R
\]
being suitably integrable.

Our main result states that this sublinear expectation with discrete-time volatility uncertainty converges
to the $G$-expectation as the number $n$ of periods tends to infinity, if the domain of volatility uncertainty is scaled as
$\bD/n:=\{n^{-1}\Gamma:\, \Gamma\in\bD\}$.

\begin{theorem}\label{th:limit}
  Let $\xi: \Omega\to\R$ be a continuous function satisfying $|\xi(\omega)|\leq c (1+\|\omega\|_{\infty})^p$
  for some constants $c,p>0$. Then $\widehat{\cE}^n_{\bD/n}(\xi)\to \cE_\bD(\xi)$ as $n\to\infty$; that is,
  \begin{equation}\label{eq:limitThm}
    \sup_{P\in \cP^n_{\bD/n}} E^P[\xi(\widehat{X^n})] \to \sup_{P\in \cP_{\bD}} E^P[\xi].
  \end{equation}
\end{theorem}

We shall see that all expressions in~\eqref{eq:limitThm} are well defined and finite. Moreover, we will show in Theorem~\ref{th:limitRamification} that the result also holds true for a ``strong'' formulation of volatility uncertainty.

\begin{remark}
  Theorem~\ref{th:limit} cannot be extended to the case where $\xi$ is merely in $\L^1_G$, which is defined as the completion of $C_b(\Omega;\R)$ under the norm $\|\xi\|_{L^1_G}:=\sup\{E^P|\xi|,\, P\in \cP_\bD\}$. This is because $\|\,\cdot\,\|_{L^1_G}$ ``does not see'' the discrete-time objects, as illustrated by the following example.
  Assume for simplicity that $0\notin\bD$ and let $A\subset \Omega$ be the set of paths with finite variation. Since $P(A)=0$ for any $P\in \cP_\bD$, we have $\xi:=1-\1_A=1$ in $\L^1_G$ and the right hand side of~\eqref{eq:limitThm} equals one. However, the trajectories of $\widehat{X^n}$ lie in $A$, so that $\xi(\widehat{X^n})\equiv0$ and the left hand side of~\eqref{eq:limitThm} equals zero.
\end{remark}

In view of the previous remark, we introduce a smaller space $\L^1_*$, defined as the completion of $C_b(\Omega;\R)$ under the norm
\begin{equation}\label{eq:Dnorm}
  \|\xi\|_*:= \sup_{Q\in\cQ} E^Q |\xi|,\quad \cQ:=\cP_\bD\cup \big\{P\circ(\widehat{X^n})^{-1}:\, P\in \cP^n_{\bD/n},\;n\in\N\big\}.
\end{equation}
If $\xi$ is as in Theorem~\ref{th:limit}, then $\xi\in \L^1_*$ by Lemma~\ref{le:polygrowthImpliesD} below and so the following is a generalization of Theorem~\ref{th:limit}.

\begin{corollary}\label{co:generalizedLimitThm}
  Let $\xi\in \L^1_*$. Then $\widehat{\cE}^n_{\bD/n}(\xi)\to \cE_\bD(\xi)$ as $n\to\infty$.
\end{corollary}

\begin{proof}
  This follows from Theorem~\ref{th:limit} by approximation, using that $\|\xi\|_*$ and $\sup\{E^P|\xi|:\,P\in \cP_{\bD}\}+ \sup\{E^P|\xi(\widehat{X^n})|: P\in \cP^n_{\bD/n},\,n\in\N\}$ are equivalent norms.
\end{proof}

\section{Proofs and Ramifications}\label{se:proof}

In the next two subsections, we prove separately two inequalities that jointly imply Theorem~\ref{th:limit} and a slightly stronger result, reported in Theorem~\ref{th:limitRamification}.

\subsection{First Inequality}\label{se:proofFirstIneq}

In this subsection we prove the first inequality of~\eqref{eq:limitThm}, namely that
\begin{equation}\label{eq:limitFirstInequality}
  \limsup_{n\to\infty} \sup_{P\in \cP^n_{\bD/n}} E^P[\xi(\widehat{X^n})] \leq \sup_{P\in \cP_{\bD}} E^P[\xi].
\end{equation}
The essential step in this proof is a stability result for the volatility (see Lemma~\ref{le:tightness}(ii) below); the necessary tightness follows from the compactness of $\bD$; i.e., from $R_\bD<\infty$.
We shall denote $\lambda \bD=\{\lambda \Gamma: \,\Gamma\in\bD\}$ for $\lambda\in\R$.

\begin{lemma}\label{le:discreteMomentEstimate}
  Given $p\in[1,\infty)$, there exists a universal constant $K>0$ such that for all $0\leq k\leq l\leq n$ and $P\in \cP^n_{\bD}$,
  \begin{enumerate}[topsep=3pt, partopsep=0pt, itemsep=1pt,parsep=2pt]
    \item $E^P [\sup_{k=0,\ldots,n} |X^n_k|^{2p}]\leq K (n R_\bD)^p$,
    \item $E^P |X^n_{l}-X^n_k|^4\leq K R_\bD^2 (l-k)^2$,
    \item $E^P[(X^n_l-X^n_k)(X^n_l-X^n_k)'|\cF^n_k] \in (l-k)\bD$ $P$-a.s.
  \end{enumerate}
\end{lemma}

\begin{proof}
  We set $X:=X^n$ to ease the notation.

  (i)~Let $p\in [1,\infty)$. By the Burkholder-Davis-Gundy (BDG) inequalities there exists a universal constant $C=C(p,d)$ such that
  \[
    E^P \bigg[\sup_{k=0,\ldots,n} |X^n_k|^{2p}\bigg] \leq C E^P \|[X]_n\|^p.
  \]
  In view of $P\in \cP^n_{\bD}$, we have
  $\|[X]_n\|=\|\sum_{i=1}^n \Delta X_i (\Delta X_i)'\| \leq nd^2R_\bD$ $P$-a.s.

  (ii)~The BDG inequalities yield a universal constant $C$ such that
  \[
    E^P |X_l-X_k|^4 \leq C E^P \|[X]_l-[X]_k\|^2.
  \]
  Similarly as in (i), $P\in \cP^n_\bD$ implies that $\|[X]_l-[X]_k\|\leq (l-k)d^2R_\bD$ $P$-a.s.

  (iii)~The orthogonality of the martingale increments yields that
  \[
   E^P[(X_l-X_k)(X_l-X_k)'|\cF^n_k] = \sum_{i=k+1}^l E^P[\Delta X_i(\Delta X_i)'|\cF^n_k].
  \]
  Since $E^P[\Delta X_i(\Delta X_i)'|\cF^n_{i-1}]\in \bD$ $P$-a.s.\ and since $\bD$ is convex,
  \[
    E^P[\Delta X_i(\Delta X_i)'|\cF^n_k]=E^P\big[E^P[\Delta X_i(\Delta X_i)'|\cF^n_{i-1}]\big|\cF^n_k\big]
  \]
  again takes values in $\bD$. It remains to observe that if $\Gamma_1,\dots,\Gamma_m\in \bD$, then $\Gamma_1+\dots+\Gamma_m\in m \bD$ by convexity.
\end{proof}

The following lemma shows in particular that all expressions in Theorem~\ref{th:limit} are well defined and finite.

\begin{lemma}\label{le:finiteness}
  Let $\xi: \Omega\to\R$ be as in Theorem~\ref{th:limit}. Then $\|\xi\|_*<\infty$; that is,
  \begin{equation}\label{eq:uniformInt}
    \sup_{n\in\N} \sup_{P\in \cP^n_{\bD/n}} E^P|\xi(\widehat{X^n})|<\infty\quad\mbox{and}\quad\sup_{P\in \cP_{\bD}} E^P|\xi|<\infty.
  \end{equation}
\end{lemma}

\begin{proof}
  Let $n\in\N$ and $P\in \cP^n_{\bD/n}$. By the assumption on $\xi$, there exist constants $c,p>0$ such that
  \begin{align*}
    E^P|\xi(\widehat{X^n})| \leq c+c E^P \bigg[\sup_{0\leq t\leq T} |\widehat{X^n_t}|^p\bigg]\leq c+cE^P \bigg[\sup_{k=0,\dots,n} |X^n_k|^p\bigg].
  \end{align*}
  Hence Lemma~\ref{le:discreteMomentEstimate}(i) and the observation that $R_{\bD/n}=R_\bD/n$ yield that $E^P|\xi(\widehat{X^n})|\leq K R_\bD^{p/2}$ and the first claim follows.
  The second claim similarly follows from the estimate that
  $E^P[\sup_{0\leq t\leq T} |B_t|^p]\leq C_p$ for all $P\in\cP_\bD$, which is obtained from the BDG inequalities by using that
  $\bD$ is bounded.
\end{proof}

We can now prove the key result of this subsection.

\begin{lemma}\label{le:tightness}
  For each $n\in\N$, let $\{M^n=(M^n_k)_{k=0}^n,\tilde{P}^n\}$ be a martingale with law $P^n\in \cP^n_{\bD/n}$ on $(\R^d)^{n+1}$ and let
  $Q^n$ be the law of $\widehat{M^n}$ on $\Omega$. Then
  \begin{enumerate}[topsep=3pt, partopsep=0pt, itemsep=1pt,parsep=2pt]
    \item the sequence $(Q^n)$ is tight on $\Omega$,
    \item any cluster point of $(Q^n)$ is an element of $\cP_{\bD}$.
  \end{enumerate}
\end{lemma}

\begin{proof}
  (i)~Let $0\leq s\leq t\leq T$. As $R_{\bD/n}=R_\bD/n$, Lemma~\ref{le:discreteMomentEstimate}(ii) implies that
  \[
    E^{Q^n} |B_t-B_s|^4 = E^{\tilde{P}^n} |\widehat{M^n_t}-\widehat{M^n_s}|^4\leq C |t-s|^2
  \]
  for a constant $C>0$. Hence $(Q^n)$ is tight by the moment criterion.

  (ii)~Let $Q$ be a cluster point, then $B$ is a $Q$-martingale as a consequence of the uniform integrability implied by Lemma~\ref{le:discreteMomentEstimate}(i) and it remains to show that
  $d\br{B}_t/dt \in\bD$ holds $Q\times dt$-a.e.
  It will be useful to characterize $\bD$ by scalar inequalities: given $\Gamma\in\S^d$, the separating hyperplane theorem implies that
  \begin{equation}\label{eq:separatingHyperplane}
   \Gamma\in \bD \q\mbox{if and only if}\q \ell(\Gamma)\leq C^\ell_{\bD}:=\sup_{A\in \bD} \ell(A)\quad \mbox{for all}\quad \ell \in (\S^d)^*,
  \end{equation}
  where $(\S^d)^*$ is the set of all linear functionals $\ell:\S^d\to \R$.

  Let $H: [0,T]\times \Omega\to [0,1]$ be a continuous and adapted function
  and let $\ell \in (\S^d)^*$. We fix $0\leq s<t\leq T$ and denote $\Delta_{s,t}Y:=Y_t-Y_s$ for a process $Y=(Y_u)_{0\leq u\leq T}$.
  Let $\eps>0$ and let $\tilde{\bD}$ be any neighborhood of $\bD$, then for $n$ sufficiently large,
  \[
    E^{\tilde{P}^n}\Big[(\Delta_{s,t}\widehat{M^n})(\Delta_{s,t}\widehat{M^n})'\Big|\sigma\big(\widehat{M^n_u},0\leq u \leq s-\eps\big)\Big]\in (t-s)\widetilde{\bD}\quad \tilde{P}^n\as
  \]
  as a consequence of Lemma~\ref{le:discreteMomentEstimate}(iii).
  Since $\tilde{\bD}$ was arbitrary, it follows by~\eqref{eq:separatingHyperplane} that
  \begin{align*}
    &\limsup_{n\to\infty}  E^{Q^n} \big[H(s-\eps,B)\,\big\{\ell\big((\Delta_{s,t}B)(\Delta_{s,t}B)'\big)-C^\ell_\bD (t-s)\big\}\big] \\
      &\;= \limsup_{n\to\infty} E^{\tilde{P}^n} \big[H(s-\eps,\widehat{M^n})\,\big\{\ell\big((\Delta_{s,t}\widehat{M^n})(\Delta_{s,t}\widehat{M^n})'\big)-C^\ell_\bD(t-s)\big\}\big] \leq 0.
  \end{align*}
  Using~\eqref{eq:uniformInt} with $\xi(\omega)=\|\omega\|^2_\infty$, we may pass to the limit and conclude that
  \begin{equation}\label{eq:proofVolatilityStab}
    E^Q \big[H(s-\eps,B)\,\ell\big((\Delta_{s,t}B)(\Delta_{s,t}B)'\big)\big]
    \leq E^Q \big[H(s-\eps,B)\,C^\ell_\bD (t-s)\big].
  \end{equation}
  Since $H(s-\eps,B)$ is $\cF_s$-measurable and
  \[
    E^Q [(\Delta_{s,t}B)(\Delta_{s,t}B)'|\cF_s]
    =E^Q [B_tB_t' - B_sB_s'|\cF_s]=E^{Q}[\br{B}_t-\br{B}_s |\cF_s]
  \]
  as $B$ is a square-integrable $Q$-martingale,~\eqref{eq:proofVolatilityStab} is equivalent to
  \[
    E^Q \big[H(s-\eps,B)\,\ell\big(\br{B}_t-\br{B}_s\big)\big]
    \leq E^Q \big[H(s-\eps,B)\,C^\ell_\bD (t-s)\big].
  \]
  Using the continuity of $H$ and dominated convergence as $\eps\to0$, we obtain
  \[
    E^Q \big[H(s,B)\,\ell\big(\br{B}_t-\br{B}_s\big)\big]
    \leq E^Q \big[H(s,B)\,C^\ell_\bD (t-s)\big]
  \]
  and then it follows that
  \[
    E^Q\bigg[\int_0^T H(t,B)\,\ell(d\br{B}_t)\bigg] \leq E^Q\bigg[\int_0^T H(t,B)C^\ell_\bD\,dt\bigg].
  \]
  By an approximation argument, this inequality extends to functions $H$ which are measurable instead of continuous.
  It follows that $\ell(d\br{B}_t/dt) \leq C^\ell_\bD$ holds $Q\times dt$-a.e., and since $\ell\in (\S^d)^*$ was arbitrary,~\eqref{eq:separatingHyperplane} shows that $d\br{B}_t/dt\in\bD$ holds $Q\times dt$-a.e.
\end{proof}

We can now deduce the first inequality of Theorem~\ref{th:limit} as follows.

\begin{proof}[Proof of~\eqref{eq:limitFirstInequality}]
  Let $\xi$ be as in Theorem~\ref{th:limit} and let $\eps>0$. For each $n\in\N$ there exists an $\eps$-optimizer $P^n\in \cP^n_{\bD/n}$; i.e., if
  $Q^n$ denotes the law of $\widehat{X^n}$ on $\Omega$ under $P_n$, then
  \[
    E^{Q^n}[\xi] = E^{P^n}[\xi(\widehat{X^n})]\geq \sup_{P\in \cP^n_{\bD/n}} E^P[\xi(\widehat{X^n})] - \eps.
  \]
  By Lemma~\ref{le:tightness}, the sequence $(Q^n)$ is tight and any cluster point belongs to $\cP_\bD$. Since $\xi$
  is continuous and~\eqref{eq:uniformInt} implies $\sup_n E^{Q_n}|\xi|<\infty$, tightness yields that $\limsup_n E^{Q^n}[\xi]\leq \sup_{P\in \cP_{\bD}} E^P[\xi]$. Therefore,
  \[
    \limsup_{n\to\infty} \sup_{P\in \cP^n_{\bD/n}} E^P[\xi(\widehat{X^n})] \leq \sup_{P\in \cP_{\bD}} E^P[\xi]+\eps.
  \]
  Since $\eps>0$ was arbitrary, it follows that~\eqref{eq:limitFirstInequality} holds.
\end{proof}

Finally, we also prove the statement preceding Corollary~\ref{co:generalizedLimitThm}.

\begin{lemma}\label{le:polygrowthImpliesD}
  Let $\xi: \Omega\to\R$ be as in Theorem~\ref{th:limit}. Then $\xi\in\L^1_*$.
\end{lemma}

\begin{proof}
  We show that $\xi^m:=(\xi\wedge m) \vee m$ converges to $\xi$ in the norm $\|\cdot\|_*$ as $m\to\infty$, or equivalently, that the upper expectation $\sup \{E^Q[\,\cdot\,]:\, Q\in \cQ\}$ is continuous along the decreasing sequence $|\xi-\xi^m|$, where $\cQ$ is as in~\eqref{eq:Dnorm}.
  Indeed, $\cQ$ is tight by (the proof of) Lemma~\ref{le:tightness}. Using that $\|\xi\|_*<\infty$ by Lemma~\ref{le:finiteness}, we can then argue as in the proof of~\cite[Theorem~12]{DenisHuPeng.2010} to obtain the claim.
\end{proof}

\subsection{Second Inequality}\label{se:proofSecondIneq}

The main purpose of this subsection is to show the second inequality ``$\geq$'' of~\eqref{eq:limitThm}. Our proof will yield a more precise version of Theorem~\ref{th:limit}. Namely, we will include ``strong'' formulations of volatility uncertainty both in discrete and in continuous time; i.e., consider laws generated by integrals with respect to a fixed random walk (resp.\ Brownian motion). In the financial interpretation, this means that the uncertainty can be generated by \emph{complete} market models.

\paragraph{Strong Formulation in Continuous Time.}
Here we shall consider \emph{Brownian} martingales: with $P_0$ denoting the Wiener measure, we define
\[
  \cQ_\bD=\bigg\{P_0\circ \Big(\int f(t,B)\,dB_t\Big)^{-1}:\; f\in C\big([0,T]\times \Omega;\sqrt{\bD}\big) \mbox{ adapted} \bigg\},
\]
where $\sqrt{\bD}=\{\sqrt{\Gamma}:\,\Gamma\in\bD\}$. (For $\Gamma\in\S^d_+$, $\sqrt{\Gamma}$ denotes the unique square-root in $\S^d_+$.)
We note that $\cQ_\bD$ is a (typically strict) subset of $\cP_\bD$. The elements of $\cQ_\bD$ with nondegenerate $f$ have the predictable representation property; i.e., they correspond to a complete market in the terminology of mathematical finance. We have the following density result; the proof is deferred to the end of the section.

\begin{proposition}\label{pr:density}
  The convex hull of $\cQ_D$ is a weakly dense subset of $\cP_\bD$.
\end{proposition}

We can now deduce the connection between $\cE_\bD$ and the $G$-expectation associated with $\bD$.

\begin{remark}\label{rk:GexpIdentification}
  (i) Proposition~\ref{pr:density} implies that
  \begin{equation}\label{eq:strongWeak}
    \sup_{P\in \cQ_{\bD}} E^P[\xi] = \sup_{P\in \cP_{\bD}} E^P[\xi],\quad \xi\in C_b(\Omega;\R).
  \end{equation}
  In \cite[Section~3]{DenisHuPeng.2010} it is shown that the $G$-expectation as introduced in~\cite{Peng.07, Peng.08} coincides with
  the mapping $\xi\mapsto \sup_{P\in \cQ^*_{\bD}} E^P[\xi]$ for a certain set $\cQ^*_{\bD}$ satisfying $\cQ_{\bD}\subseteq \cQ^*_{\bD} \subseteq \cP_{\bD}$. In particular, we deduce that the right hand side of~\eqref{eq:strongWeak} is indeed equal to the $G$-expectation, as claimed in Section~\ref{se:main}.

  (ii) A result similar to Proposition~\ref{pr:density} can also be deduced from~\cite[Proposition 3.4.]{SonerTouziZhang.2010rep}, which relies on a PDE-based verification argument of stochastic control. We include a (possibly more enlightening) probabilistic proof at the end of the section.
\end{remark}

\paragraph{Strong Formulation in Discrete Time.} For fixed $n\in\N$, we consider
\[
  \Omega_n:=\big\{\omega=(\omega_1,\dots,\omega_n):\, \omega_i\in \{1,\dots,d+1\},\; i=1,\dots,n\big\}
\]
equipped with its power set and let $P_n:=\{(d+1)^{-1},\dots,(d+1)^{-1}\}^n$ be the product probability associated with the uniform distribution. Moreover, let $\xi_1,\dots,\xi_n$ be an i.i.d.\ sequence of $\R^d$-valued random variables on $\Omega_n$ such that $|\xi_k|=d$ and such that the components of $\xi_k$ are orthonormal in $L^2(P_n)$, for each $k=1,\ldots,n$. Let $Z_k=\sum_{l=1}^k \xi_l$ be the associated random walk. Then, we consider martingales $M^f$ which are discrete-time integrals of $Z$ of the form
\[
  M^f_k = \sum_{l=1}^{k} f(l-1,Z) \Delta Z_l,
\]
where $f$ is measurable and adapted  with respect to the filtration generated by $Z$; i.e., $f(l,Z)$ depends only on $Z|_{\{0,\dots,l\}}$. We define
\[
  \cQ^n_\bD\hspace{-1pt}=\hspace{-1pt}\bigg\{\hspace{-1pt}P_n\circ (M^f)^{-1};\; f\hspace{-1pt}: \{0,\dots,n-1\}\times (\R^d)^{n+1}\hspace{-1pt}\to\sqrt{\bD}\mbox{ measurable, adapted} \hspace{-1pt}\bigg\}.
\]
To see that $\cQ^n_\bD\subseteq \cP^n_\bD$, we note that
$\Delta_{k}M^f = f(k-1,Z)\xi_{k}$ and the orthonormality property of $\xi_{k}$ yield
\[
  E^{P_n}\big[ \Delta_{k}M^f\big(\Delta_{k}M^f\big)' \big| \sigma(Z_1,\dots,Z_{k-1})\big] = f(k-1,Z)^2\in\bD\quad P_n\as,
\]
while $|\xi_k|=d$ and $f^2\in\bD$ imply that
\[
  \big\|\Delta_{k}M^f\big(\Delta_{k}M^f\big)'\big\|=|f(k-1,Z)\xi_{k}|^2 \in \big[d^2r_\bD,d^2R_\bD\big]\quad P_n\as
\]

\begin{remark}\label{rk:constructionOfXi}
 We recall from~\cite{He.90} that such $\xi_1,\dots,\xi_n$ can be constructed as follows. Let $A$ be an orthogonal $(d+1)\times(d+1)$ matrix whose last row is $((d+1)^{-1/2},\dots,(d+1)^{-1/2})$ and let $v_l\in\R^{d}$ be column vectors such that $[v_1,\dots,v_{d+1}]$ is the matrix obtained from $A$ by deleting the last row. Setting $\xi_k(\omega):=(d+1)^{1/2}v_{\omega_k}$ for $\omega=(\omega_1,\dots,\omega_n)$ and $k=1,\ldots,n$, the above requirements are satisfied.
\end{remark}

We can now formulate a result which includes Theorem~\ref{th:limit}.

\begin{theorem}\label{th:limitRamification}
  Let $\xi: \Omega\to\R$ be as in Theorem~\ref{th:limit}. Then
  \begin{align}\label{eq:limitThmRamification}
    \lim_{n\to\infty} \sup_{P\in \cQ^n_{\bD/n}} E^P[\xi(\widehat{X^n})]
     &= \lim_{n\to\infty} \sup_{P\in \cP^n_{\bD/n}} E^P[\xi(\widehat{X^n})] \nonumber\\
     &=\sup_{P\in \cQ_{\bD}} E^P[\xi] \nonumber \\
     &=\sup_{P\in \cP_{\bD}} E^P[\xi].
  \end{align}
\end{theorem}

\begin{proof}
  Since $\cQ^n_{\bD/n}\subseteq \cP^n_{\bD/n}$ for each $n\geq1$, the inequality~\eqref{eq:limitFirstInequality} yields that
  \[
    \limsup_{n\to\infty} \sup_{P\in \cQ^n_{\bD/n}} E^P[\xi(\widehat{X^n})] \leq \sup_{P\in \cP_{\bD}} E^P[\xi].
  \]
  As the equality in~\eqref{eq:limitThmRamification} follows from Proposition~\ref{pr:density}, it remains to show that
  \[
    \liminf_{n\to\infty} \sup_{P\in \cQ^n_{\bD/n}} E^P[\xi(\widehat{X^n})]  \geq \sup_{P\in \cQ_{\bD}} E^P[\xi].
  \]
  To this end, let $P\in\cQ_{\bD}$; i.e., $P$ is the law of a martingale of the form
  \[
    M=\int f(t,W)\,dW_t,
  \]
  where $W$ is a Brownian motion and $f\in C\big([0,T]\times \Omega;\sqrt{\bD}\big)$ is an adapted function. We shall construct
  martingales $M^{(n)}$ whose laws are in $\cQ^n_{\bD/n}$ and tend to $P$.

  For $n\geq1$, let $Z^{(n)}_k=\sum_{l=1}^k \xi_l$ be the random walk on $(\Omega_n,P_n)$ as introduced before Remark~\ref{rk:constructionOfXi}.
  Let
  \[
    W^{(n)}_t:= n^{-1/2} \sum_{k=1}^{[nt/T]} \xi_k,\quad 0\leq t\leq T
  \]
  be the piecewise constant c\`adl\`ag version of the scaled random walk and let $\hat{W}^{(n)}:=n^{-1/2}\widehat{Z^{(n)}}$ be its continuous counterpart obtained by linear interpolation. It follows from the central limit theorem that
  \[
    \big(W^{(n)},\hat{W}^{(n)}\big)\Rightarrow (W,W)\quad\mbox{on}\quad D([0,T];\R^{2d}),
  \]
  the space of c\`adl\`ag paths equipped with the Skorohod topology.
  Moreover, since $f$ is continuous, we also have that
  \[
    \big(W^{(n)},f\big([nt/T]T/n,\hat{W}^{(n)}\big)\big)\Rightarrow (W,f(t,W))\quad\mbox{on}\quad D([0,T];\R^{d+d^2}).
  \]
  Thus, if we introduce the discrete-time integral
  \[
    M^{(n)}_k:= \sum_{l=1}^k f\big((l-1)T/n,\hat{W}^{(n)}\big) \Big(\hat{W}^{(n)}_{lT/n} - \hat{W}^{(n)}_{(l-1)T/n}\Big),
  \]
  it follows from the stability of stochastic integrals (see \cite[Theorem~4.3 and Definition~4.1]{DuffieProtter.92}) that
  \[
    \Big(M^{(n)}_{[nt/T]}\Big)_{0\leq t\leq T} \Rightarrow M\quad\mbox{on}\quad D([0,T];\R^d).
  \]
  Moreover, since the increments of $M^{(n)}$ uniformly tend to $0$ as $n\to\infty$, it also follows that
  \[
    \widehat{M}^{(n)}\Rightarrow M\quad\mbox{on}\quad \Omega.
  \]
  As $f^2/n$ takes values in $\bD/n$, the law of $M^{(n)}$ is contained in $\cQ^n_{\bD/n}$ and the proof is complete.
\end{proof}

It remains to give the proof of Proposition~\ref{pr:density}, which we will obtain by a randomization technique. Since similar arguments, at least for the scalar case, can be found elsewhere (e.g., \cite[Section~5]{Kusuoka.95}), we shall be brief.

\begin{proof}[Proof of Proposition~\ref{pr:density}]
  We may assume without loss of generality that
  \begin{equation}\label{eq:nondegen}
    \mbox{there exists an invertible element }\Gamma_*\in\bD.
  \end{equation}
  Indeed, using that $\bD$ is a convex subset of $\S^d_+$, we observe that~\eqref{eq:nondegen} is equivalent
  to $K=\{0\}$ for $K:=\bigcap_{\Gamma\in\bD} \ker \Gamma$. If $k=\dim K>0$, a change of coordinates bring us to the situation where $K$ corresponds to
  the last $k$ coordinates of $\R^d$. We can then reduce all considerations to $\R^{d-k}$ and thereby recover the situation of~\eqref{eq:nondegen}.

  \emph{1.~Regularization.} We first observe that the set
  \begin{equation}\label{eq:regularization}
    \big\{P\in\cP_\bD:\,d\br{B}_t/dt \geq \eps\idmat_d\; P\times dt\ae\mbox{ for some }\eps>0\big\}
  \end{equation}
  is weakly dense in $\cP_\bD$. (Here $\idmat_d$ denotes the unit matrix.) Indeed, let $M$ be a martingale whose law is in $\cP_\bD$. Recall~\eqref{eq:nondegen} and let $N$ be an independent continuous Gaussian martingale with $d\br{N}_t/dt=\Gamma_*$. For $\lambda\uparrow 1$, the law of $\lambda M + (1-\lambda)N$ tends to the law of $M$ and is contained in the set~\eqref{eq:regularization}, since $\bD$ is convex.

  \emph{2.~Discretization.} Next, we reduce to martingales with piecewise constant volatility. Let $M$ be a martingale whose law belongs to~\eqref{eq:regularization}.
  We have
  \[
    M= \int \sigma_t\,dW_t\quad\mbox{for}\quad \sigma_t:=\sqrt{d\br{M}/dt} \quad\mbox{and}\quad W:=\int\sigma^{-1}_t\,dM_t,
  \]
  where $W$ is a Brownian motion by L\'evy's theorem. For $n\geq 1$, we introduce
  $
    M^{(n)}= \int \sigma^{(n)}_t\,dW_t,
  $
  where $\sigma^{(n)}$ is an $\S^d_+$-valued piecewise constant process satisfying
  \[
    \big(\sigma^{(n)}_t\big)^2 = \Pi_\bD \bigg[\bigg(\frac{n}{T}\int_{(k-1)T/n}^{kT/n} \sigma_s\,ds\bigg)^2\bigg],\quad t\in \big(kT/n,(k+1)T/n\big]
  \]
  for $k=1,\dots,n-1$, where $\Pi_\bD: \S^d\to\bD$ is the Euclidean projection.
  On $[0,T/n]$ one can take, e.g., $\sigma^{(n)}:=\sqrt{\Gamma_*}$. We then have
  \[
    E\big\|\big\langle M-M^{(n)}\big\rangle_T\big\|=E\int_0^T \big\|\sigma_t-\sigma^{(n)}_t\big\|^2\,dt\to0
  \]
  and in particular $M^{(n)}$ converges weakly to $M$.

  \emph{3.~Randomization}. Consider a martingale of the form $M= \int \sigma_t\,dW_t$, where $W$ is a Brownian motion on some given filtered probability space and $\sigma$ is an adapted $\sqrt{\bD}$-valued process which is piecewise constant; i.e.,
  \[
    \sigma=\sum_{k=0}^{n-1} \1_{[t_k, t_{k+1})} \sigma(k)\quad \mbox{for some}\quad0=t_0<t_1< \dots<t_n=T
  \]
  and some $n\geq1$.
  Consider also a second probability space carrying a Brownian motion $\tilde{W}$ and a
  sequence $U^1,\dots, U^n$ of $\R^{d\times d}$-valued random variables such that
  the components $\{U^k_{ij}:\, 1\leq i,j\leq d;\,1\leq k\leq n\}$ are i.i.d.\ uniformly distributed on $(0,1)$ and independent of $\tilde{W}$.

  Using the existence of regular conditional probability distributions, we can construct functions
  $\Theta_k: C([0,t_k];\R^d)\times(0,1)^{d^2}\times\dots\times (0,1)^{d^2}\to \sqrt{\bD}$ such that the random variables
  $
    \tilde{\sigma}(k):= \Theta_k(\tilde{W}|_{[0,t_k]}, U^1,\dots,U^k)
  $
  satisfy
  \begin{equation}\label{eq:discretrizeLaw}
   \big\{\tilde{W},\tilde{\sigma}(0), \dots,\tilde{\sigma}(n-1)\big\} = \big\{W,\sigma(0), \dots,\sigma(n-1)\big\}\quad\mbox{in law.}
  \end{equation}
  We can then consider the volatility corresponding to a fixed realization of $U^1,\dots, U^n$. Indeed, for
  $u=(u^1,\dots,u^n)\in (0,1)^{nd^2}$, let
  \[
    \tilde{\sigma}(k;u):= \Theta_k\big(\tilde{W}|_{[0,t_k]}, u^1,\dots,u^k\big)
  \]
  and consider $\tilde{M}^u=\int \tilde{\sigma}^u_t\,d\tilde{W}_t$, where
  $\tilde{\sigma}^u:=\sum_{k=0}^{n-1} \1_{[t_k,t_{k+1})} \tilde{\sigma}(k;u)$.
  For any $F\in C_b\big(\Omega;\R)$, the equality~\eqref{eq:discretrizeLaw} and Fubini's theorem yield that
  \begin{align*}
    E[F(M)] = E\big[F\big(\tilde{M}^{(U^1,\dots,U^n)}\big)\big]
    &= \int_{(0,1)^{nd^2}} E[F(\tilde{M}^u)]\,du \\
    &\leq \sup_{u\in (0,1)^{nd^2}} E[F(\tilde{M}^u)].
  \end{align*}
  Hence, by the Hahn-Banach theorem, the law of $M$ is contained in the weak closure of the convex hull of the laws of
  $\{\tilde{M}^u:\, u\in (0,1)^{nd^2}\}$. We note that $\tilde{M}^u$ is of the form $\tilde{M}^u=\int g(t,\tilde{W})\,d\tilde{W}_t$ with a measurable, adapted, $\sqrt{\bD}$-valued function $g$, for each fixed $u$.

  \emph{4.~Smoothing.} As $\cQ_D$ is defined through continuous functions, it remains to approximate $g$ by a continuous function $f$.
  Let $g: [0,T]\times \Omega\to\sqrt{\bD}$ be a measurable adapted function and $\delta>0$. By standard density arguments there exists $\tilde{f}\in C\big([0,T]\times \Omega;\S^d\big)$ such that
  \[
    E\int_0^T \|\tilde{f}(t,\tilde{W})-g(t,\tilde{W})\|^2\,dt\leq\delta.
  \]
  Let $f(t,x):=\sqrt{\Pi_\bD\big(\tilde{f}(t,x)^2\big)}$. Then $f\in C\big([0,T]\times \Omega;\sqrt{\bD}\big)$ and
  \[
    \|f-g\|^2 \leq \|f^2-g^2\|\leq \|\tilde{f}^2-g^2\| \leq (\|\tilde{f}\|+\|g\|)\|\tilde{f}-g\|\leq 2\sqrt{R_\bD}\,\|\tilde{f}-g\|
  \]
  (see \cite[Theorem~X.1.1]{Bhatia.97} for the first inequality). By Jensen's inequality we conclude that $E\int_0^T \|f(t,\tilde{W})-g(t,\tilde{W})\|^2\,dt \leq 2 \sqrt{T R_\bD\delta}$, which, in view of the above steps, completes the proof.
\end{proof}



\newcommand{\dummy}[1]{}

\end{document}